\documentclass[12pt,oneside]{article}

\usepackage{latexsym,amssymb,amsfonts,amsmath,amsthm,ifthen}

\makeatletter
\newcommand{\subjclass}[2][2010]{%
  \let\@oldtitle\@title%
  \gdef\@title{\@oldtitle\footnotetext{#1 \emph{Mathematics subject classification.} #2}}%
}
\newcommand{\keywords}[1]{%
  \let\@@oldtitle\@title%
  \gdef\@title{\@@oldtitle\footnotetext{\emph{Key words and phrases.} #1.}}%
}
\makeatother

\theoremstyle{plain}
\newtheorem{thm}{Theorem}[section]
\newtheorem{cor}[thm]{Corollary}
\newtheorem{lem}[thm]{Lemma}
\newtheorem{prop}[thm]{Proposition}

\theoremstyle{definition}

\theoremstyle{remark}
\newtheorem{rem}{Remark}[section]
\newtheorem{ex}{Example}[section]

\newcommand{\vtsp}{\hspace{0.1em}}

\DeclareMathOperator{\Dom}{Dom}
\DeclareMathOperator{\Curl}{curl}
\DeclareMathOperator{\Div}{div}
\DeclareMathOperator{\re}{Re}

\newcommand{\loc}{\mathrm{loc}}
\newcommand{\supp}{\mathrm{supp}}

\newcommand{\R}{\mathbb{R}}
\newcommand{\Z}{\mathbb{Z}}
\newcommand{\N}{\mathbb{N}}
\newcommand{\C}{\mathbb{C}}
\renewcommand{\epsilon}{\varepsilon}

\newcommand{\sfrac}[2]{\text{\small$\dfrac{#1}{#2}$}}

\renewcommand{\S}{\mathbb{S}}  

\newcommand{\ri}{\mathrm{i}}  
\newcommand{\rd}{\mathrm{d}}  

\newcommand{\Dirac}[1][A]{\mathcal D_{#1}}  

\newcommand{\magp}{A}  
\newcommand{\bmagp}{A'}  
\newcommand{\magf}{B}  

\newcommand{\bfa}{\mathbf{a}}  
\newcommand{\bfb}{\mathbf{b}}
\newcommand{\bfc}{\mathbf{c}}
\newcommand{\bfx}{\mathbf{x}}
\newcommand{\euv}[1]{\mathbf{e}_{#1}}  

\newcommand{\cKf}{X}  
\newcommand{\cX}{Y}  
\newcommand{\XxcX}{Z}  
\newcommand{\wei}[1][{1}]{\ifthenelse{\equal{#1}{1}}{\vtsp\omega}{\vtsp\omega^{#1}}}  
\newcommand{\udf}{U}  
\newcommand{\rof}{V}  
\newcommand{\crf}[1][\mu]{W_{#1}}  

\newcommand{\nsi}{\alpha}  
\newcommand{\nsj}{\beta}  
\newcommand{\nsk}{\gamma}  

\newcommand{\icT}{\widehat{\cKf}}  
\newcommand{\icN}{\widehat{\XxcX}}  
\newcommand{\icB}{\widehat{\cX}}  
\newcommand{\Pvec}{P}  
\newcommand{\Nvec}{N}  
\newcommand{\pP}{\mathcal P}  
\newcommand{\ues}[1]{\theta_{#1}}  

\newcommand{\OpQ}{Q}
\newcommand{\OpS}{S}
\newcommand{\proj}[1]{\Pi_{#1}}  
\newcommand{\odD}[1]{T_{#1}}  

\newcommand{\nzX}[1][\cKf]{\Omega_{#1}}  
\newcommand{\sfX}{C^{\infty}_{\cKf}}  
\newcommand{\cssfX}{C^{\infty}_{\cKf,0}}  
\newcommand{\Rc}{\rho}  
\newcommand{\pcf}[1][\rho]{\ifthenelse{\equal{#1}{}}{\xi}{\xi_{#1}}}  
\newcommand{\Rcf}[1][\rho]{\ifthenelse{\equal{#1}{}}{\eta}{\eta_{#1}}}  
\newcommand{\ecf}[1][\epsilon]{\ifthenelse{\equal{#1}{}}{\zeta}{\zeta_{#1}}}  

\newcommand{\Lspin}{\mathcal H}  
\newcommand{\wLspin}{\mathcal H_{\cKf}}  

\newcommand{\ipd}[2]{\langle{#1},{#2}\rangle}
\newcommand{\bigipd}[2]{\bigl\langle{#1},\vtsp{#2}\bigr\rangle}
\newcommand{\wipd}[2]{\langle{#1},{#2}\rangle_{\cKf}}

\newcommand{\abs}[1]{\lvert{#1}\rvert}
\newcommand{\lrabs}[1]{\left\lvert{#1}\right\rvert}
\newcommand{\bigabs}[1]{\bigl\lvert{#1}\bigr\rvert}
\newcommand{\norm}[1]{\lVert{#1}\rVert}

\newcommand{\wnorm}[1]{\lVert{#1}\rVert_{\cKf}}

\title{On the non-existence of zero modes}
\author{Daniel M.~Elton}

\begin{document}

\maketitle

\begin{abstract}
We consider magnetic fields on $\R^3$ which are parallel to a conformal Killing field. 
When the latter generates a simple rotation we show that a Weyl-Dirac operator with such a magnetic field
cannot have a zero mode. In particular this allows us to expand the class of non zero mode producing 
magnetic fields to include examples of non-trivial smooth compactly supported fields. 
\end{abstract}

\section{Introduction}

Given a magnetic potential $\magp$ on $\R^3$ we can consider the Weyl-Dirac operator
\[
\Dirac=\sigma.(-\ri\nabla-\magp)
\]
where $\sigma=(\sigma_1,\sigma_2,\sigma_3)$ are the Pauli matrices
and $\nabla=(\nabla_1,\nabla_2,\nabla_3)$ is the gradient operator. 
The operator $\Dirac$ acts on \emph{spinors}, or $\C^2$ valued functions on $\R^3$. 
A standard construction (see Section \ref{sec:regZMs} for further details) 
shows that $\Dirac$ initially defined on $C^\infty_0(\R^3,\C^2)$
has a unique (unbounded) self-adjoint extension on $\Lspin=L^2(\R^3,\C^2)$; we will use the same notation 
for both operators. A \emph{zero mode} is any non-trivial $\psi\in\Lspin$ which solves
$\Dirac\psi=0$; in other words, a zero mode is an eigenfunction of $\Dirac$ with eigenvalue $0$. 

The magnetic field corresponding to $\magp$ is $\magf=\Curl\magp$. If two potentials $\magp$ and $\magp'$ generate 
the same field the corresponding Weyl-Dirac operators are unitarily (gauge) equivalent. 
In particular the existence of zero modes is determined by $\magf$.

Zero modes have been studied in a number of contexts in mathematical physics including 
the stability of matter (\cite{FLL}, \cite{LY}) and chiral gauge theories (\cite{AMN1}, \cite{AMN2}). 
Most early work concentrated on the construction of explicit examples, including the original example (\cite{LY}), 
examples with arbitrary multiplicity (\cite{AMN2}), compact support (\cite{E1}) 
and a certain geometric property (\cite{ES}; further details below). 
Some subsequent work moved toward studying the set of all zero mode producing potentials (or fields) 
within a given class of potentials. 
More precisely, for relatively short range potentials the set of zero mode producing potentials is
nowhere dense (\cite{BE1}, \cite{BE2}) 
and is generically a co-dimension $1$ sub-manifold 
(\cite{E2}; slightly different classes of potentials were considered in these works). 
In particular, if $\magp\in C^\infty\cap L^3$ the set of scalings $t\in\R$ for which 
$\Dirac[t\magp]$ has a zero mode is a discrete subset of $\R$.

To further our understanding of the set of zero mode producing potentials 
we can consider the problem in various asymptotic regimes. 
The strong field regime corresponds to large scalings $t\to+\infty$;  
a simple rescaling of the zero mode equation shows this also corresponds to the semi-classical regime.
In this context a bound can be established for the rate at which zero mode producing scalings occur (\cite{ET}),
while the leading order asymptotics are known for certain $\magp$'s (\cite{E3}). 
Our ultimate goal is to obtain the relevant leading order asymptotics for general $\magp$; 
in the present work we provide some strong limitations on
the form such asymptotics can take by giving examples of non-trivial $\magp\in C^\infty_0$ for which 
$\Dirac[t\magp]$ does not have zero modes for \emph{any} $t\in\R$.
While it is known that uni-directional magnetic fields cannot support zero modes in general (\cite{FLL}),
our examples appear to be the first known classes of localised zero mode free potentials (or fields) which are invariant under scaling. 

We consider magnetic potentials for which the corresponding magnetic field is everywhere 
parallel to a \emph{conformal Killing field}. The latter is 
a vector field $X$ which is the infinitesimal generator of a
conformal symmetry. For a general metric this is equivalent to the condition $\mathcal L_Xg=\lambda g$
where $g$ is the Riemanninan metric, $\mathcal L_X$ denotes the Lie derivative, 
and $\lambda$ is a scalar function.  
On $\R^3$ (with the standard Euclidean metric) this condition becomes
\begin{equation}
\label{eq:cKfdefpr}
\nabla_i\cKf_j+\nabla_j\cKf_i=\sfrac23\vtsp(\Div\cKf)\vtsp\delta_{ij}.
\end{equation}
Conformal Killing fields on $\R^3$ form a $10$-dimensional space 
(see Remark \ref{rem:cKfR3} and Proposition \ref{prop:gencKfchar} below). 
When $\magf=\Curl\magp$ is parallel to such a field $\Dirac$ possess extra 
symmetry properties which can be exploited to gain information about its spectrum. 
Indeed, the construction of examples of zero mode producing fields in \cite{ES}
is based on an application of these ideas to a special conformal Killing field. 
By contrast, in the present work we show that $\Dirac$ does not have a zero mode if $\magf$
is parallel to a conformal Killing field satisfying the extra condition
\begin{equation}
\label{eq:simrotcond}
\cKf.\Curl\cKf=0.
\end{equation}
Such conformal Killing fields form a $7$-dimensional algebraic subset of 
the space of all conformal Killing fields (see Proposition \ref{prop:charcKfXcX0}).

\begin{thm}
\label{thm:mainthm}
Suppose $\cKf$ is a non-trivial conformal Killing field on $\R^3$ which satisfies \eqref{eq:simrotcond}.
Let $\magp$ be a smooth magnetic potential for which the corresponding magnetic field 
$\Curl\magp$ is everywhere parallel to $\cKf$. 
Then the Weyl-Dirac operator $\Dirac$ does not have a zero mode.
\end{thm}

\begin{rem}
Clearly any constant vector field is a conformal Killing field satisfying \eqref{eq:simrotcond}. 
The previously known fact that Weyl-Dirac operators corresponding to uni-directional magnetic fields cannot have 
zero modes (\cite{FLL}) is thus a special case of Theorem \ref{thm:mainthm}. 
\end{rem}

If $f\in C^\infty(\R^2)$ then $\magp(\bfx)=f(x_1^2+x_2^2,x_3)\vtsp(0,0,1)$ is a smooth magnetic potential 
with $\magf=\Curl\magp=-2f_1(x_1^2+x_2^2,x_3)\cKf$ 
where $f_1$ is the derivative of $f$ with respect to its first variable and $\cKf=(-x_2,x_1,0)$
is a conformal Killing field satisfying \eqref{eq:simrotcond}. 
It follows that $\Dirac$ cannot have a zero mode. The class of such fields is 
invariant under scaling. Further details are given in Example \ref{ex:ro} below; 
another explicit class of fields is discussed in Example \ref{ex:cr}.

\begin{cor}
\label{cor:cinfty0potnZMs}
There exist non-trivial $\magp\in C^\infty_0$ such that $\Dirac[tA]$ does not have a zero mode 
for any $t\in\R$.
\end{cor}

\begin{rem}
\label{rem:cKfR3}
The space of conformal Killing fields is the Lie algebra corresponding to the Lie group of conformal transformations.
On $\R^3$ Liouville's Theorem on conformal mappings shows the latter is the same as the group of M\"{o}bius transformations 
(see \cite{Bl}, for example). In turn, this is just the (identity component of the) 
$10$-dimensional indefinite orthogonal group $O(1,4)$ (see \cite{Be}; this group is also the symmetry group of de Sitter space). 
Below we give an elementary direct argument to determine the conformal Killing fields on $\R^3$ and identify 
those which satisfy \eqref{eq:simrotcond} (see Propositions \ref{prop:gencKfchar} and \ref{prop:charcKfXcX0}).
\end{rem}

\begin{rem}
\label{rem:spheredis}
Stereographic projection gives a conformal equivalence between the $3$-dimensional sphere $\S^3$ and $\R^3$;
the corresponding spaces of conformal Killing fields are then isomorphic. It follows that rotations of $\S^3$, 
or equivalently $\R^4$, give rise to conformal Killing fields on $\R^3$. 
A general rotation of $\R^4$ can be decomposed into independent rotations in a pair of orthogonal $2$-dimensional planes.
An \emph{isoclinic rotation} of $\R^4$ is one in which the angles of the independent rotations are equal;
the corresponding conformal Killing fields on $\R^3$ underlie the work of \cite{ES}.
On the other hand, a \emph{simple rotation} of $\R^4$ is one in which one of the independent rotations is trivial;
the corresponding conformal Killing fields on $\R^3$ satisfy \eqref{eq:simrotcond}.
See Remark \ref{rem:isosimrot} for some further discussion.

The conformal invariance of the Dirac operator means we could view our results and arguments 
in the setting of $\S^3$. Such a viewpoint is partially adopted in \cite{ES} and \cite{E3}, 
but requires a substantial amount of differential geometric language. 
In the present work we aim to keep arguments more accessible by working in $\R^3$
using direct calculations. 
\end{rem}

Key to our argument is the fact that, when $\magf$ is parallel to a conformal Killing field $\cKf$,
it is possible to differentiate spinors in the direction parallel to $\cKf$ in a way that commutes with $\Dirac$;
this directional derivative is given by the operator $\OpQ$ defined at the start of Section \ref{sec:symm}, while the 
commutativity property is given in Proposition \ref{prop:baspropDQS}. 
When $\cKf$ satisfies \eqref{eq:simrotcond} and our spinor is a zero mode of $\Dirac$ we can further show this directional 
derivative must be $0$; Proposition \ref{prop:normspindecomp} gives the necessary norms bounds, 
while regularity issues associated to the applicability of these bounds are dealt with in Section \ref{sec:regZMs}.
On the other hand, by studying $\OpQ$ on the closed integral curves of $\cKf$, we can show this operator cannot have 
$0$ as an eigenvalue; see Proposition \ref{prop:specQ}. Preliminaries for the analysis of $\OpQ$ along the integral curves
are covered in Section \ref{sec:srot}. Various identities relating to conformal Killing fields are collected in Section \ref{sec:cKfgen},
while in Section \ref{sec:CharAdmFields} we explicitly determine the fields satisfying \eqref{eq:simrotcond} and eliminate those which cannot
be parallel to a magnetic field. The various pieces of our argument are pulled together in Section \ref{sec:pfthm1}.

\subsubsection*{Notation}

In general vectors in $\R^3$ are denoted by upper case or bold lower case letters;
for example $X=(X_1,X_2,X_3)$ and $\bfx=(x_1,x_2,x_3)$. 
Also $\euv{3}=(0,0,1)$ denotes the unit vector parallel to the $x_3$-axis.
The Roman letters $i,j,k,\dots$ are used for indices taking values in $\{1,2,3\}$; 
repeated indices imply summation over this set.
For any vector field $X$ we set $\nabla_X=X.\nabla=X_i\nabla_i$. 
Let $\delta_{ij}$ denote the Kronecker delta and $\epsilon_{ijk}$ the totally anti-symmetric tensor;
in particular $\sigma_i\sigma_j=\delta_{ij}I_2+\ri\epsilon_{ijk}\sigma_k$, so
\begin{equation}
\label{eq:PauliMatId}
(\bfa.\sigma)(\bfb.\sigma)=\bfa.\bfb I_2+\ri(\bfa\times\bfb).\sigma,
\quad\text{$\bfa,\bfb\in\R^3$.}
\end{equation}
Note that $\epsilon_{ijk}\epsilon_{lmk}=\delta_{il}\delta_{jm}-\delta_{im}\delta_{jl}$,
which leads to the vector triple product formula
\begin{equation}
\label{eq:vectripleprod}
\bfa\times(\bfb\times\bfc)=(\bfa.\bfc)\bfb-(\bfa.\bfb)\bfc,
\quad\text{$\bfa,\bfb,\bfc\in\R^3$.}
\end{equation}

We use $\abs{\cdot}$ for the standard norm in any finite dimensional vector space. 
The standard inner products in $\R^3$ and $\C^2$ will be denoted by a dot and 
$\ipd{\cdot}{\cdot}$ respectively.
For a sequence $\{\psi_n\}_{n\in\N}$ indexed by $\N$ we simply write $\psi_n\to\psi$ 
to indicate a limit as $n\to\infty$.

\section{Conformal Killing fields}
\label{sec:cKfgen}

Let $\cKf$ be a conformal Killing field on $\R^3$. 
Set $\wei=\abs{\cKf}$, $\cX=\Curl\cKf$ and $\XxcX=\cKf\times\cX$.
In this section we collect some identities for $\cKf$, $\cX$ and $\XxcX$ which 
follow directly from \eqref{eq:cKfdefpr}. 
A number of these involve powers of $\wei$; in such cases we consider the identities on the open set 
\begin{equation}
\label{eq:nzscKf}
\nzX=\R^3\setminus\{\bfx\in\R^3:\wei(\bfx)=0\}.
\end{equation}

From \eqref{eq:cKfdefpr} we get
\begin{equation}
\label{eq:XDabsXa}
\nabla_\cKf\wei[\alpha]
=\sfrac{\alpha}{2}\wei[\alpha-2]\nabla_\cKf\wei[2]
=\sfrac{\alpha}{2}\wei[\alpha-2]\cKf_i\cKf_j(\nabla_i\cKf_j+\nabla_j\cKf_i)
=\sfrac{\alpha}{3}(\Div\cKf)\wei[\alpha]
\end{equation}
for any $\alpha\in\R$.
Multiplying \eqref{eq:cKfdefpr} by $\cKf_i$ and summing over $i$ also leads to
\begin{equation}
\label{eq:XDXX2DivX}
\nabla_\cKf\cKf=-\sfrac12\nabla\wei[2]+\sfrac23(\Div\cKf)\cKf,
\end{equation}
and thus
\begin{equation}
\label{eq:XxDX2}
(\cKf\times\nabla)\wei[2]=-2\cKf\times\nabla_\cKf\cKf.
\end{equation}
Using \eqref{eq:vectripleprod} 
\begin{equation}
\label{eq:XcZexpr}
\cKf\times\XxcX=(\cKf.\cX)\cKf-\wei[2]\cX.
\end{equation}
Since $\cX=\nabla\times\cKf$ we can use the same identity and \eqref{eq:XDXX2DivX} to get
\begin{equation}
\label{eq:XxCurlX}
\XxcX=\cKf\times\cX=\cKf_i\nabla\cKf_i-\nabla_\cKf\cKf
=-2\nabla_\cKf\cKf+\sfrac23(\Div\cKf)\cKf,
\end{equation}
so $\cKf\times\XxcX=-2\cKf\times\nabla_\cKf\cKf=(\cKf\times\nabla)\wei[2]$
with the help of \eqref{eq:XxDX2}. 
Comparing this with \eqref{eq:XcZexpr} we now get
\begin{equation}
\label{eq:altCurlX}
(\cKf\times\nabla)\wei
=\sfrac12\wei[-1](\cKf\times\nabla)\wei[2]
=\sfrac12\wei[-1](\cKf.\cX)\cKf-\sfrac12\wei\cX.
\end{equation}
We can also rewrite \eqref{eq:XxCurlX} to get
\begin{equation}
\label{eq:genexprXDX}
\nabla_\cKf\cKf=\sfrac13(\Div\cKf)\cKf-\sfrac12\XxcX.
\end{equation}
Comparison with \eqref{eq:XDXX2DivX} then leads to
\begin{equation}
\label{eq:genexprDX2}
\nabla\wei[2]
=\sfrac23(\Div\cKf)\cKf+\XxcX.
\end{equation}

Using \eqref{eq:cKfdefpr} twice we get
\begin{align}
&\nabla_i\cX_j+\nabla_j\cX_i
=\epsilon_{jkl}\nabla_k\nabla_i\cKf_l+\epsilon_{ikl}\nabla_k\nabla_j\cKf_l\nonumber\\
&\qquad{}=-\epsilon_{jkl}\nabla_k\nabla_l\cKf_i+\sfrac23\epsilon_{jkl}\nabla_k(\Div\cKf)\delta_{il}
-\epsilon_{ikl}\nabla_k\nabla_l\cKf_j+\sfrac23\epsilon_{ikl}\nabla_k(\Div\cKf)\delta_{jl}\nonumber\\
\label{eq:CurlXKf}
&\qquad{}=\sfrac23\epsilon_{jki}\nabla_k(\Div\cKf)+\sfrac23\epsilon_{ikj}\nabla_k(\Div\cKf)=0,
\end{align}
while $\Div\cX=\Div\Curl\cKf=0$. It follows that $\cX$ is a Killing field 
(that is, satisfies \eqref{eq:cKfdefpr} with $0$ on the right hand side).

Now suppose $\{\nsi,\nsj,\nsk\}=\{1,2,3\}$. 
Then, with no summation over indices,
\[
Y_{\nsi}\bigl(\nabla_{\nsk}\cKf_{\nsi}-\nabla_{\nsi}\cKf_{\nsk}\bigr)
=\epsilon_{\nsi\nsj\nsk}Y_{\nsi}Y_{\nsj}
=\bigl(\nabla_{\nsj}\cKf_{\nsk}-\nabla_{\nsk}\cKf_{\nsj}\bigr)Y_{\nsj}.
\]
Hence
\[
Y_{\nsi}\nabla_{\nsk}\cKf_{\nsi}+Y_{\nsj}\nabla_{\nsk}\cKf_{\nsj}+Y_{\nsk}\nabla_{\nsk}\cKf_{\nsk}
=Y_{\nsi}\nabla_{\nsi}\cKf_{\nsk}+Y_{\nsj}\nabla_{\nsj}\cKf_{\nsk}+Y_{\nsk}\nabla_{\nsk}\cKf_{\nsk}
\]
or, returning to the summation convention,
\begin{equation}
\label{eq:CurlXDX}
\cX_i\nabla\cKf_i=\nabla_\cX\cKf.
\end{equation}
Combined with \eqref{eq:CurlXKf} we then get
\[
\nabla(\cKf.\cX)
=\cX_i\nabla\cKf_i+\cKf_i\nabla\cX_i
=\nabla_\cX\cKf-\nabla_\cKf\cX
\]
(the Lie bracket of $\cKf$ and $\cX$).
On the other hand, \eqref{eq:CurlXDX} and \eqref{eq:cKfdefpr} give
\[
\nabla_\cX\cKf_k
=\cX_i\nabla_k\cKf_i
=-\nabla_\cX\cKf_k+\sfrac23(\Div\cKf)\cX_k.
\]
Thus $\nabla_\cX\cKf=\frac13(\Div\cKf)\cX$ and hence
\begin{equation}
\label{eq:genLXCurlX}
\nabla_\cKf\cX=\sfrac13(\Div\cKf)\cX-\nabla(\cKf.\cX).
\end{equation}

By \eqref{eq:cKfdefpr}
\begin{equation}
\label{eq:DeltacKf}
\Delta\cKf_j=\nabla_i\nabla_i\cKf_j
=-\nabla_i\nabla_j\cKf_i+\sfrac23\nabla_j(\Div\cKf)
=-\sfrac13\nabla_j(\Div\cKf),
\end{equation}
and hence
\[
\nabla_k\cX_l
=\epsilon_{lmn}\nabla_m\nabla_k\cKf_n
=-\epsilon_{lmn}\nabla_m\nabla_n\cKf_k+\sfrac23\epsilon_{lmk}\nabla_m(\Div\cKf)
=-2\epsilon_{klm}\Delta\cKf_m.
\]
Thus
\begin{equation}
\label{eq:XDCurlXXxDelX}
\nabla_\cKf\cX
=2\cKf\times\Delta\cKf.
\end{equation}
Since $\XxcX_k=\epsilon_{kij}\cKf_i\cX_j$ and $\epsilon_{kij}\epsilon_{klm}=\delta_{il}\delta_{jm}-\delta_{im}\delta_{jl}$ 
we also get
\begin{equation}
\label{eq:XxCurlXDCurlX}
\nabla_\XxcX\cX_l
=-2\cKf_l\cX_m\Delta\cKf_m+2\cKf_m\Delta\cKf_m\cX_l
=2(\cKf.\Delta\cKf)\cX_l-2(\cX.\Delta\cKf)\cKf_l.
\end{equation}

\section{Simple rotations}
\label{sec:srot}

We now suppose the conformal Killing field $\cKf$ satisfies \eqref{eq:simrotcond}, namely $\cKf.\cX=0$.
Then $\cKf$, $\cX$ and $\XxcX$ are mutually orthogonal while
\begin{equation}
\label{eq:nXxcXXcZY}
\abs{\XxcX}=\wei\abs{\cX}\quad\text{and}\quad
\cKf\times\XxcX=-\wei[2]\cX
\end{equation}
by \eqref{eq:XcZexpr}.
Thus the normalised vector fields 
$\icT=\wei[-1]\cKf$, $\icB=\abs{\cX}^{-1}\cX$ and $\icN=\abs{\XxcX}^{-1}\XxcX$
provide an orthonormal frame on $\nzX\cap\nzX[\cX]$. 
Also \eqref{eq:genexprDX2} gives
\begin{equation}
\label{eq:simexprDX2}
\abs{\nabla\wei}^2
=\sfrac14\wei^{-2}\abs{\nabla\wei[2]}^2
=\sfrac19(\Div\cKf)^2+\sfrac14\abs{\cX}^2,
\end{equation}
while \eqref{eq:XDCurlXXxDelX} and \eqref{eq:genLXCurlX} lead to
\[
\cKf\times\Delta\cKf=\sfrac12\nabla_\cKf\cX=\sfrac16(\Div\cKf)\cX
\quad\Longrightarrow\quad
\cKf\times(\cKf\times\Delta\cKf)=\sfrac16(\Div\cKf)\XxcX.
\]
However $\cKf\times(\cKf\times\Delta\cKf)=(\cKf.\Delta\cKf)\cKf-\wei[2]\Delta\cKf$ by \eqref{eq:vectripleprod}, so
\[
\wei[2]\cX.\Delta\cKf
=\cX.\Bigl[(\cKf.\Delta\cKf)\cKf-\sfrac16(\Div\cKf)\XxcX\Bigr]
=0.
\]
Thus \eqref{eq:genLXCurlX} and \eqref{eq:XxCurlXDCurlX} simplify to
\begin{equation}
\label{eq:DXYDZY}
\nabla_\cKf\cX=\sfrac13(\Div\cKf)\cX
\quad\text{and}\quad
\nabla_{\XxcX}\cX
=2(\cKf.\Delta\cKf)\cX,
\end{equation}
so $\nabla_\cKf\cX$ and $\nabla_\XxcX\cX$ are parallel to $\cX$. It follows that
$\nabla_{\icT}\icB=0=\nabla_{\icN}\icB$ on $\nzX\cap\nzX[\cX]$.
If $\Pvec\in\R^3$ we have $\Pvec=(\Pvec.\icT)\icT+(\Pvec.\icN)\icN+(\Pvec.\icB)\icB$, and hence
\begin{equation}
\label{eq:DPhatY}
\nabla_{\Pvec}\icB=(\Pvec.\icB)\nabla_{\icB}\icB.
\end{equation}
Note that $\nabla_{\icB}\icB$ is smooth on $\nzX\cap\nzX[\cX]$ so bounded by a constant
$C_K$ on any compact set $K\subset\nzX\cap\nzX[\cX]$.

\begin{lem}
\label{lem:simcKfplanic}
Let $\bfx_0\in\nzX\cap\nzX[\cX]$. Set $\Nvec=\icB(\bfx_0)$ and $\pP=\{\bfx\in\R^3:(\bfx-\bfx_0).\Nvec=0\}$ 
(the plane through $\bfx_0$ with normal $\Nvec$).
Then $\cKf(\bfx).\Nvec=0$ for all $\bfx\in\pP$.
\end{lem}

\begin{proof}
Suppose $\Pvec.\Nvec=0$ for some $\Pvec\in\R^3$. Set $\gamma(t)=\bfx_0+t\Pvec\in\pP$ for $t\in\R$, 
and choose $\tau>0$ so that $K=\gamma([0,\tau])\subset\nzX\cap\nzX[\cX]$.
By \eqref{eq:DPhatY}
\begin{equation}
\label{eq:ddsNinPi}
\frac{\rd}{\rd t}\icB(\gamma(t))=\nabla_{\gamma'(t)}\icB(\gamma(t))
=\Pvec.\icB(\gamma(t))\,\nabla_{\icB}\icB(\gamma(t))
\end{equation}
and hence
\[
\lrabs{\frac{\rd}{\rd t}\Pvec.\icB(\gamma(t))}\le C_K\abs{\Pvec}\bigabs{\Pvec.\icB(\gamma(t))}.
\]
Since $\Pvec.\icB(\gamma(0))=\Pvec.\Nvec=0$ it follows that $\Pvec.\icB(\gamma(t))=0$ for all $t\in[0,\tau]$. 
Hence the right hand side of \eqref{eq:ddsNinPi} is $0$, and so $\icB(\bfx)=\icB(\gamma(0))=\Nvec$ 
for all $\bfx\in\gamma([0,\tau])$. This clearly extends to all $\bfx$ in $\pP_0$, the path connected component of
$\nzX\cap\nzX[\cX]\cap\pP$ containing $\bfx_0$. 
Thus $\Pvec.\cX=0$ on $\pP_0$. However $\Pvec.\cX$ is harmonic 
(this follows easily from the fact that $\cX$ is a Killing field). Unique continuation (see \cite[Theorem XIII.63]{RS4}, for example) now implies 
$\Pvec.\cX=0$ on $\pP$. Since this holds for all $\Pvec\in\R^3$ with $\Pvec.\Nvec=0$ we must have that $\cX$ 
is parallel to $\Nvec$ on $\pP$, so $\cKf.\Nvec=0$ on $\nzX[\cX]\cap\pP$ by \eqref{eq:simrotcond}.
However $\cKf.\Nvec$ is smooth while $\nzX[\cX]\cap\pP$ is dense in $\pP$ 
(again by unique continuation). Thus $\cKf.\Nvec=0$ on $\pP$.
\end{proof}

We now consider integral curves of the vector field $\cKf$ 
(these are the magnetic field lines when $\magf$ is parallel to $\cKf$).
Suppose $\gamma:I\to\R^3$ is such a curve. Then $\gamma'(t)=\cKf(\gamma(t))$
and $\gamma''(t)=(\rd/\rd t)\cKf(\gamma(t))=\nabla_\cKf\cKf(\gamma(t))$ so,
using \eqref{eq:genexprXDX} and \eqref{eq:nXxcXXcZY},
\begin{equation}
\label{eq:altCurlXcp}
\gamma'(t)\times\gamma''(t)
=\cKf\times\nabla_\cKf\cKf
=-\sfrac12\cKf\times\XxcX
=\sfrac12\wei[2]\cX.
\end{equation}
Likewise $(\rd/\rd t)\cX(\gamma(t))=\nabla_\cKf\cX(\gamma(t))=f(t)\,\cX(\gamma(t))$
where $f(t)=\frac13\Div\cKf(\gamma(t))$ by \eqref{eq:DXYDZY}.
As $f$ is smooth it follows that, on $\gamma$, $\cX$ is either nowhere $0$  
or identically equal to $0$. In the latter case \eqref{eq:altCurlXcp} then implies $\gamma''(t)$
is parallel to $\gamma'(t)$; hence $\gamma$ is a straight line so, in particular, 
cannot be closed.

Now suppose $\gamma$ is a closed integral curve of $\cKf$ with (minimal) period $\tau$. 
Thus $\gamma$ lies in $\nzX\cap\nzX[\cX]$. The fact that $(\rd/\rd t)\cX(\gamma(t))$ 
is parallel to $\cX(\gamma(t))$ then implies $\icB$ is constant on $\gamma$.
Denote this constant vector by $\Nvec$ and let $\pP$ be as given in Lemma \ref{lem:simcKfplanic}
for some $\bfx_0$ lying on $\gamma$. It follows that $\gamma$ lies in the plane $\pP$.

We need the integrals of three quantities around $\gamma$.
Using \eqref{eq:XDabsXa} we get
\[
(\Div\cKf)(\gamma(t))
=3(\wei[-1]\nabla_\cKf\wei)(\gamma(t))
=3\abs{\gamma'(t)}^{-1}\sfrac{\rd}{\rd t}\abs{\gamma'(t)}=3\sfrac{\rd}{\rd t}\log\abs{\gamma'(t)},
\]
so
\begin{equation}
\label{eq:0intDivXloop}
\int_0^\tau(\Div\cKf)(\gamma(t))\,\rd t=0.
\end{equation}
Next, the (signed) curvature of $\gamma$ is given as
\[
\kappa(t)=\abs{\gamma'(t)}^{-3}\Nvec.(\gamma'(t)\times\gamma''(t))
=\sfrac12(\wei[-1]\icB.\cX)(\gamma(t))=\sfrac12(\wei[-1]\abs{\cX})(\gamma(t)).
\]
Since we traverse the simple closed loop $\gamma$ once for $t\in[0,\tau]$ 
the Hopf Umlaufsatz (see \cite[Theorem 1.7]{SII}, for example) gives
$\int_0^\tau \kappa(t)\,\abs{\gamma'(t)}\,\rd t=\pm2\pi$. It follows that
\begin{equation}
\label{eq:HUevalint}
\int_0^\tau\abs{\cX(\gamma(t))}\,\rd t=4\pi.
\end{equation}
Finally suppose $\Curl\magp$ is everywhere parallel to $\cKf$ for some smooth $\magp$. 
Let $\Omega$ denote the region of the plane $\pP$ bounded by $\gamma$. 
Using the Kelvin-Stokes theorem we get
\[
\int_0^\tau\cKf.\magp(\gamma(t))\,\rd t
=\int_0^\tau\gamma'(t).\magp(\gamma(t))\,\rd t
=\oint_{\gamma}\magp.\rd\gamma
=\iint_{\Omega}(\Curl\magp).\rd\Omega.
\]
Now $\rd\Omega$ is parallel to $\Nvec$ while $\Curl\magp$ is parallel to $\cKf$. 
However $\cKf.\Nvec=0$ on $\pP$ by Lemma \ref{lem:simcKfplanic}. Thus
\begin{equation}
\label{eq:cloloopflux0}
\int_0^\tau\cKf.\magp(\gamma(t))\,\rd t=0.
\end{equation}
(This shows the overall flux of the magnetic field $\magf=\Curl\magp$
through $\gamma$ is zero.)

\section{Characterisation and admissible fields}
\label{sec:CharAdmFields}

We begin by explicitly identifying all conformal Killing fields on $\R^3$. 
Clearly the set of conformal Killing fields is a linear space. 

\begin{prop}
\label{prop:gencKfchar}
On $\R^3$ the space of conformal Killing fields is $10$-dimensional.
More precisely, $\cKf$ is a conformal Killing field iff
\begin{equation}
\label{eq:gencKfR3}
\cKf=\bfa+b_0\bfx+\bfb\times\bfx+(\bfc.\bfx)\bfx-\sfrac12\abs{\bfx}^2\bfc
\end{equation}
for some $b_0\in\R$ and $\bfa,\bfb,\bfc\in\R^3$.
\end{prop}

\begin{proof}
A straightforward calculation shows that anything of the form \eqref{eq:gencKfR3} is a conformal Killing field.

Allowing all choices $\{\nsi,\nsj,\nsk\}=\{1,2,3\}$
\eqref{eq:cKfdefpr} is equivalent to
\begin{subequations}
\label{eq:ord1X}
\begin{align}
\label{eq:ord1aa}
\nabla_\nsi\cKf_\nsi&=\nabla_\nsj\cKf_\nsj,\\
\label{eq:ord1ab}
\nabla_\nsi\cKf_\nsj&=-\nabla_\nsj\cKf_\nsi.
\end{align}
\end{subequations}
For second order derivatives we can combine \eqref{eq:ord1aa} and \eqref{eq:ord1ab} to get
\begin{subequations}
\label{eq:ord2X}
\begin{equation}
\label{eq:ord2aaa}
\nabla_\nsi^2\cKf_\nsi=\nabla_\nsi\nabla_\nsj\cKf_\nsj=-\nabla_\nsj^2\cKf_\nsi.
\end{equation}
Repeated use of \eqref{eq:ord1ab} gives 
$\nabla_\nsi\nabla_\nsj\cKf_\nsk=-\nabla_\nsi\nabla_\nsk\cKf_\nsj=\nabla_\nsj\nabla_\nsk\cKf_\nsi=-\nabla_\nsi\nabla_\nsj\cKf_\nsk$
so
\begin{equation}
\label{eq:ord2abc}
\nabla_\nsi\nabla_\nsj\cKf_\nsk=0.
\end{equation}
\end{subequations}
Combining \eqref{eq:ord2aaa} and \eqref{eq:ord1aa} we get 
$\nabla_\nsi^3\cKf_\nsi=-\nabla_\nsi\nabla_\nsj^2\cKf_\nsi=-\nabla_\nsj^3\cKf_\nsj$.
Since this holds for all choices of $\nsi$ and $\nsj$ we must have
\[
\nabla_\nsi^3\cKf_\nsi=\nabla_\nsi^2\nabla_\nsj\cKf_\nsj=0.
\]
Using \eqref{eq:ord1aa} and \eqref{eq:ord2abc} we also get
\[
\nabla_\nsi^2\nabla_\nsj\cKf_\nsi=\nabla_\nsi\nabla_\nsj\nabla_\nsk\cKf_\nsk=0.
\]
It follows that all third order derivatives of $\cKf$ are identically $0$. 
Hence $\cKf$ is completely determined by the values of its first two derivatives at $0$.
Set 
\[
a_\nsi=\cKf_\nsi(0),\quad
b_0=\nabla_\nsi\cKf_\nsi(0),\quad
b_\nsi=\epsilon_{\nsi\nsj\nsk}\nabla_\nsj\cKf_\nsk(0)
\quad\text{and}\quad
c_\nsi=\nabla_\nsi^2\cKf_\nsi(0).
\]
Note that, \eqref{eq:ord1X} shows the definitions of $b_0$ and $\bfb$ are consistent, 
while the remaining second order derivatives can be expressed in terms of $\bfc$ using \eqref{eq:ord2X}.
It is straightforward to check that $\cKf$ is now given by \eqref{eq:gencKfR3}.
\end{proof}

It is now straightforward to identify which conformal Killing fields satisfy \eqref{eq:simrotcond}.

\begin{prop}
\label{prop:charcKfXcX0}
The conformal Killing fields satisfying \eqref{eq:simrotcond} are those given as
\[
\cKf=\bfa+b_0\bfx,\quad
\cKf=\bfb\times(\bfx-\bfx_0)
\quad\text{or}\quad
\cKf=\nu\bfc+\bfc.(\bfx-\bfx_0)(\bfx-\bfx_0)-\sfrac12\abs{\bfx-\bfx_0}^2\bfc
\]
for some $\bfa,\bfb,\bfc,\bfx_0\in\R^3$ and $b_0,\nu\in\R$.
\end{prop}

\begin{proof}
Suppose $\cKf$ is given by \eqref{eq:gencKfR3}. Then $\cX=2\bfb+2\bfc\times\bfx$ so
\begin{align*}
\cKf.\cX
&=2\bfa.\bfb+2\bfa.(\bfc\times\bfx)+2b_0\bfb.\bfx+2(\bfb\times\bfx).(\bfc\times\bfx)+2(\bfc.\bfx)(\bfb.\bfx)-\abs{\bfx}^2\bfb.\bfc\\
&=2\bfa.\bfb+2(\bfa\times\bfc+b_0\bfb).\bfx+(\bfb.\bfc)\abs{\bfx}^2
\end{align*}
Thus \eqref{eq:simrotcond} is equivalent to 
\[
\bfa.\bfb=0=\bfc.\bfb
\quad\text{and}\quad
b_0\bfb=\bfc\times\bfa.
\]
The first case arises if $\bfb=\bfc=0$. If $\bfc=0$ but $\bfb\neq0$ then we must have
$b_0=0$ and $\bfa.\bfb=0$ so $\bfa=-\bfb\times\bfx_0$ where $\bfx_0=\abs{\bfb}^{-2}\bfb\times\bfa$; 
this corresponds to the second case.

Now suppose $\bfc\neq0$. Set $\bfx_0=\abs{\bfc}^{-2}(\bfc\times\bfb-b_0\bfc)$
so $b_0=-\bfc.\bfx_0$ and $\bfb=-\bfc\times\bfx_0$ since $\bfb.\bfc=0$. 
Hence $\abs{\bfc}^2\abs{\bfx_0}^2=b_0^2+\abs{\bfb}^2$ while
$\bfc\times(\bfa+b_0\bfx_0)=0$, leading to
\[
0=\bfc\times\bigl(\bfc\times(\bfa+b_0\bfx_0)\bigr)
\quad\Longrightarrow\quad
\bfa=\abs{\bfc}^{-2}(\bfc.\bfa)\bfc-\abs{\bfc}^{-2}b_0^2\bfc-b_0\bfx_0.
\]
Then
\begin{align*}
&\bfc.(\bfx-\bfx_0)(\bfx-\bfx_0)-\sfrac12\abs{\bfx-\bfx_0}^2\bfc\\
&\qquad{}=(\bfc.\bfx_0)\bfx_0-\sfrac12\abs{\bfx_0}^2\bfc
-\bigl((\bfc.\bfx)\bfx_0+(\bfc.\bfx_0)\bfx-(\bfx_0.\bfx)\bfc\bigr)
+(\bfc.\bfx)\bfx-\sfrac12\abs{\bfx}^2\bfc\\
&\qquad{}=-b_0\bfx_0-\sfrac12\abs{\bfc}^{-2}(\abs{\bfb}^2+b_0^2)\bfc
+b_0\bfx+\bfb\times\bfx+(\bfc.\bfx)\bfx-\sfrac12\abs{\bfx}^2\bfc
\end{align*}
The third case formula for $\cKf$ now follows if we take
$\nu=\sfrac12\abs{\bfc}^{-2}\bigl(2\bfa.\bfc+\abs{\bfb}^2-b_0^2\bigr)$.
\end{proof}

Several of the fields given in Proposition \ref{prop:charcKfXcX0} have isolated fixed points. 
In these cases (almost all) integral curves converge to these fixed points as $t\to+\infty$ 
or $t\to-\infty$; this can be seen by considering the possible cases.

If $b_0\neq0$ in the first type of field given in Proposition \ref{prop:charcKfXcX0} we can write $\cKf=b_0(\bfx-\bfx_0)$
with $\bfx_0=-b_0^{-1}\bfa$. This (monopole) field has a single isolated fixed point at $\bfx_0$,
with $\gamma(t)\to0$ as either $t\to-\infty$ or $t\to+\infty$ for all integral curves $\gamma$
(the sign depends on the sign of $b_0$).

If $\nu<0$ in the third type of field then $\cKf$ has isolated fixed points at 
$\bfx_0\pm\abs{\bfc}^{-1}\abs{2\nu}^{1/2}\bfc$.
All integral curves converge to one fixed point as $t\to-\infty$ and the other as $t\to+\infty$, 
with the exception of those curves which are half lines along the axis parallel to $\bfc$ and passing through $\bfx_0$ 
(these approach a fixed point in one direction and have finite time blow up in the other). 

When $\nu=0$ these isolated fixed points merge into a single fixed point at $\bfx_0$ (dipole field). 
All integral curves converge to $\bfx_0$ in both directions, with the exception of those along 
the axis parallel to $\bfc$ and passing through $\bfx_0$.

\medskip

In the next result we show that if almost all integral curves of $\cKf$ converge to a fixed point 
then there are no smooth non-trivial magnetic fields which are everywhere parallel to $\cKf$.
We can view this as a consequence of the fact that magnetic fields are divergence free so 
cannot contain any magnetic monopoles.

\begin{lem}
\label{lem:nofixedptB}
Let $\cKf$ be a non-trivial conformal Killing field and suppose that for a dense set of points $\bfx$ in $\R^3$, 
the integral curve of $\cKf$ passing through $\bfx$ converges (to a point in $\R^3$) as either $t\to+\infty$ or $t\to-\infty$.
If $\magf$ is a smooth magnetic field which is everywhere parallel to $\cKf$ then we must have $\magf\equiv0$. 
\end{lem}

\begin{proof}
We can write $\magf=f\cKf$ for some $f\in\sfX$.
Using \eqref{eq:XDabsXa} we get
\[
\nabla_\cKf(f\wei[3])
=\wei[3]\bigl(\nabla_\cKf f+f\Div\cKf\bigr)
=\wei[3]\Div(f\cKf)
=\wei[3]\Div\magf
=0.
\]
Hence $f\wei[3]$ is constant on each integral curve of $\cKf$.

Suppose $\bfx_0$ is a fixed point of $\cKf$ and $\gamma$ is a 
non-stationary integral curve of $\cKf$ with $\gamma(t)\to\bfx_0$ as $t\to+\infty$.
For $\bfx$ near $\bfx_0$ we have $\wei(\bfx)\le C\abs{\bfx-\bfx_0}$ for some constant $C$.
Then $\wei(\gamma(t))\le C\abs{\gamma(t)-\bfx_0}$ for all sufficiently large $t$, so
\[
\wei[-2](\gamma(t))\ge C^{-2}\abs{\gamma(t)-\bfx_0}^{-2}\to+\infty
\]
as $t\to+\infty$ (not that, the left hand side is finite for all $t$ as $\gamma$ is non stationary).
However $(f\wei[3])(\gamma(t))=\beta_\gamma$ for some constant $\beta_\gamma$, so
\[
\abs{\magf(\gamma(t))}=\abs{f(\gamma(t))}\,\wei(\gamma(t))
=\abs{\beta_\gamma}\,\wei[-2](\gamma(t)),
\]
while $\magf$ is smooth, so $\magf(\gamma(t))\to\magf(\bfx_0)$ as $t\to+\infty$.
To avoid a contradiction we must take $\beta_\gamma=0$.
Hence $f(\gamma(t))=0$ for all $t$; that is, $f$ is identically $0$ on the integral curve $\gamma$.
A similar argument applies if $\gamma(t)\to\bfx_0$ as $t\to-\infty$. 

As the integral curves which converge as either $t\to+\infty$ or $t\to-\infty$ are dense in $\R^3$
we now get $f=0$ and hence $\magf=0$ on a dense subset of $\R^3$. Since $\magf$ is smooth we then get $\magf\equiv0$.
\end{proof}

As we are interested in magnetic fields which are parallel to a conformal Killing field $\cKf$ 
satisfying \eqref{eq:simrotcond}, Lemma \ref{lem:nofixedptB} shows that we may further assume $\cKf$ has no
isolated fixed points; we shall call such fields \emph{admissible}.

We consider three particular admissible fields defined as
\[
\udf=\euv{3}=(0,0,1),\qquad
\rof=\euv{3}\times\bfx=(-x_2,x_1,0)
\]
and
\[
\crf=\sfrac12\mu^2\euv{3}+(\euv{3}.\bfx)\bfx-\sfrac12\abs{\bfx}^2\euv{3}
=\sfrac12\bigl(2x_1x_3,2x_2x_3,\mu^2-x_1^2-x_2^2+x_3^2\bigr)
\]
for $\mu>0$. With an appropriate choice of origin and orientation in $\R^3$ 
any admissible field can be reduced to a scaled copy of one of these fields; 
it follows that, when considering admissible fields, it is sufficient to 
work with just these examples.

\begin{rem}
\label{rem:gpBfA'}
Suppose we can find one (smooth) magnetic potential $\bmagp$ 
with $\Curl\bmagp$ parallel to $\cKf$ and $\cKf.\bmagp=0$.
Now let $f\in C^\infty$ be any function which is constant on the integral curves of $\cKf$, 
that is $\nabla_\cKf f=0$. 
Then \eqref{eq:vectripleprod} gives
$\cKf\times(\nabla f\times\bmagp)=(\cKf.\bmagp)\nabla f-\nabla_\cKf f\bmagp=0$,
so $\nabla f\times\bmagp$ is parallel to $\cKf$. 
However, setting $\magp=f\bmagp$ we have
\[
\magf=\Curl\magp=f\Curl\bmagp+\nabla f\times\bmagp.
\]
Thus the magnetic potential $\magp=f\bmagp$ also generates a field which is parallel to $\cKf$.
\end{rem}

\smallskip
\begin{ex}
\label{ex:ud}
We have $\Div\udf=0$, $\Curl\udf=0$ and $\wei=1$ so $\nzX[\udf]=\R^3$.
The integral curves of $\udf$ are of the form $\gamma(t)=\bfx_0+t\euv{3}$ for some $\bfx_0\in\R^3$;
all are parallel to the $x_3$-axis.

Magnetic fields parallel to $\udf$ are uni-directional. 
We can clearly generate such fields by taking $\magp=f(x_1,x_2)\rof$
for some $f\in C^\infty(\R^2)$ (this corresponds to Remark \ref{rem:gpBfA'} with $\bmagp=\rof$).
\end{ex}

\smallskip
\begin{ex}
\label{ex:ro}
We have $\Div\rof=0$, $\Curl\rof=2\euv{3}$ and
$\wei[2]=x_1^2+x_2^2$, so $\nzX[\rof]=\R^3\setminus\R\euv{3}$ 
(that is, $\R^3$ with the $x_3$-axis removed). 
Note that, $\wei^{-1}\in L^1_\loc$.

It is straightforward to check that the 
integral curves are all of the form 
\[
\gamma(t)=\rho\bigl(\cos(t-t_0),\sin(t-t_0),0\bigr)+(0,0,\beta)
\]
for some $\rho\ge0$ and $t_0,\beta\in\R$. 
The integral curves are circles centred on and perpendicular to the $x_3$-axis. 
All integral curves are periodic with period $2\pi$. 
Each point on the $x_3$-axis is a stationary integral curve.

Magnetic fields parallel to $\rof$ circulate about the $x_3$-axis.
For any $\mu$, $\Curl\crf=2\rof$ and $\rof.\crf=0$; thus we can
generate such fields using Remark \ref{rem:gpBfA'} with $\bmagp=\crf$. 
Alternatively we can take $\bmagp=\udf$, as in the discussion preceding Corollary \ref{cor:cinfty0potnZMs}. 
\end{ex}

\smallskip
\begin{ex}
\label{ex:cr}
We have $\Div\crf=3x_3$, $\Curl\crf=2\euv{3}\times\bfx$ and 
\[
\wei[2]
=\sfrac14(\mu^2-\abs{\bfx}^2)^2+\mu^2x_3^2
=\sfrac14(\mu^2+\abs{\bfx}^2)^2-\mu^2(x_1^2+x_2^2),
\]
so $\nzX[\crf]=\R^3\setminus\S^1_\mu$ 
where $\S^1_\mu$ is the circle in the $x_1,x_2$-plane with centre $0$ and radius $\mu$.
Note that, $\wei^{-1}\in L^1_\loc$.

The $x_3$-axis, parametrised as $\gamma(t)=(0,0,\mu\tan(\mu(t-t_0)/2))$ for $t_0\in\R$, 
is one integral curve; note that
\[
\gamma'(t)=\sfrac12\bigl(0,0,\mu^2(1+\tan^2(\mu(t-t_0)/2))\bigr)=\crf(\gamma(t)).
\]
The remaining integral curves are circles in planes passing through the $x_3$-axis. 
Let 
\[
z(t)=r(t)+ix_3(t)=\mu\frac{1+\rho e^{-i\mu t}}{1-\rho e^{-i\mu t}}
=\mu\frac{1-\rho^2-2i\rho\sin(\mu t)}{1+\rho^2-2\rho \cos(\mu t)}
\]
for some $\rho\in[0,1)$. Then
\[
z'(t)=-2i\mu^2\frac{\rho e^{-i\mu t}}{(1-\rho e^{-i\mu t})^2}
=-\sfrac12 i(z^2(t)-\mu^2)
\]
so $r'(t)=r(t)x_3(t)$ and $x_3'(t)=\frac12(\mu^2+x_3^2(t)-r^2(t))$.
If $\theta,t_0\in\R$ it follows that
\[
\gamma(t)=r(t-t_0)(\cos\theta,\sin\theta,0)+(0,0,x_3(t-t_0))
\]
is an integral curve of $\crf$, lying in the plane which includes the $x_3$-axis 
and makes an angle of $\theta$ with the $x_1$-axis. All integral curves are periodic with period $2\pi/\mu$.

Magnetic fields parallel to $\crf$ form loops linking with the circle $\S^1_\mu$.
Let $\bmagp=(\mu^2+\abs{\bfx}^2)^{-2}\rof$. Then $\crf.\bmagp=0$ while, with the help of \eqref{eq:vectripleprod},
\[
\Curl\bmagp=2(\mu^2+\abs{\bfx}^2)^{-2}\euv{3}-4(\mu^2+\abs{\bfx}^2)^{-3}\,\bfx\times(\euv{3}\times\bfx)
=4(\mu^2+\abs{\bfx}^2)^{-3}\crf;
\]
this is clearly parallel to $\crf$. 
We can now use Remark \ref{rem:gpBfA'} to generate further magnetic fields which are parallel to $\crf$.
\end{ex}

\begin{rem}
\label{rem:isosimrot}
If we use (inverse) stereographic projection to move the fields $\rof$ and $\crf[1]$ to $\S^3$ 
we get the generators of simple rotations in orthogonal $2$-dimensional planes. 
More generally, non-constant admissible fields $\cKf$ can be linked to simple rotations of $\S^3$; 
this link can be used to give an alternative justification of the fact that the integral 
curves of $\cKf$ are planar (see the discussion after \eqref{eq:altCurlXcp}).

By contrast the conformal Killing field $\rof+\crf[1]$ corresponds to an isoclinic rotation of $\S^3$
(the integral curves are the fibres of the Hopf fibration of $\S^3$); the  zero mode examples 
considered in \cite{ES} have magnetic fields which are parallel to $\rof+\crf[1]$.
\end{rem}

\section{Symmetries}
\label{sec:symm}

Let $\cKf$ be a conformal Killing field and 
suppose $\magp$ is a magnetic potential such that $\magf=\Curl\magp$ is everywhere parallel to $\cKf$.
Introduce operators
\[
\OpQ=\cKf.(-\ri\nabla-\magp)+\sfrac14\sigma.\cX-\sfrac23\ri\Div\cKf
\quad\text{and}\quad
\OpS=\wei[-1]\sigma.\cKf;
\]
both act on spinors. 

\begin{rem}
The operator $\OpQ$ is the component of the spin connection in the direction of the magnetic field,
while $\OpS$ is the spin operator in this direction (a spinor $\psi$ has spin direction parallel to $\cKf$ 
iff it satisfies $\OpS\psi=\pm\psi$). 
\end{rem}

Recalling \eqref{eq:nzscKf}, set $\sfX=C^\infty(\nzX)$ and $\cssfX=C^\infty_0(\nzX)$; in particular, 
no restriction is placed on the limiting behaviour of functions in $\sfX$ as we approach $\{\bfx:\wei(\bfx)=0\}$,
while functions in $\cssfX$ can be extended by $0$ to give an inclusion $\cssfX\hookrightarrow C^\infty_0(\R^3)$.
We use the same notation for the spinor versions of these spaces 
(that is, $\sfX\otimes\C^2$ and $\cssfX\otimes\C^2$).

The operators $\Dirac$, $\OpQ$ and $\OpS$ satisfy some basic (anti-)commutator relations. 
It is enough to work with spinors from $\sfX$.

\begin{prop}
\label{prop:baspropDQS}
The following (anti-)commutator relations hold on $\sfX$:
\begin{enumerate}
\item[\textup{(i)}]
$[\Dirac\wei,\OpQ]=0$.
\item[\textup{(ii)}]
$[\OpQ,\OpS]=0$.
\item[\textup{(iii)}]
$\{\Dirac\wei,\OpS\}=2\OpQ+\frac12\wei[-1]\,\cKf.\cX\,\OpS$.
\end{enumerate}
\end{prop}

Note that, by $\Dirac\wei$ we mean the operator composition given by $\phi\mapsto\Dirac(\wei\phi)$.

\begin{proof}
We have $[-\ri\nabla_i-\magp_i,-\ri\nabla_j-\magp_j]=-\ri\epsilon_{ijk}\magf_k$ so
\begin{align}
[\Dirac,\cKf.(-\ri\nabla-\magp)]
&=-\ri\sigma_i(\nabla_i\cKf_j)(-\ri\nabla_j-\magp_j)+\ri\sigma.(\cKf\times\magf)\nonumber\\
\label{eq:commDQ1}
&=-\ri\sigma_i(\nabla_i\cKf_j)(-\ri\nabla_j-\magp_j)
\end{align}
since $\magf$ is parallel to $\cKf$. 
Next note that for any vectors of operators $F$ and $G$ 
\begin{equation}
\label{eq:gensigmacomm}
[\sigma.F,\sigma.G]
=\sigma_i\sigma_jF_iG_j-\sigma_j\sigma_iG_jF_i
=[F_i,G_i]+\ri\epsilon_{ijk}\sigma_k\{F_i,G_j\}.
\end{equation}
Since $\epsilon_{ijk}\epsilon_{jlm}=\delta_{kl}\delta_{im}-\delta_{km}\delta_{il}$ we get
\begin{align*}
&\epsilon_{ijk}\{-\ri\nabla_i-\magp_i,\cX_j\}
=\epsilon_{ijk}\epsilon_{jlm}\{-\ri\nabla_i-\magp_i,\nabla_l\cKf_m\}\\
&\qquad{}=2(\nabla_k\cKf_i-\nabla_i\cKf_k)(-\ri\nabla_i-\magp_i)
-\ri(\nabla_i\nabla_k\cKf_i-\nabla_i\nabla_i\cKf_k)\\
&\qquad{}=4(\nabla_k\cKf_i)(-\ri\nabla_i-\magp_i)-\sfrac43(\Div\cKf)(-\ri\nabla_k-\magp_k)-\ri\sfrac43\nabla_k(\Div\cKf)
\end{align*}
with the help of \eqref{eq:cKfdefpr}.
Combined with \eqref{eq:gensigmacomm} it follows that
\begin{align}
[\Dirac,\sigma.\cX]
&=[-\ri\nabla_i-\magp_i,\cX_i]+\ri\epsilon_{ijk}\sigma_k\{-\ri\nabla_i-\magp_i,\cX_j\}\nonumber\\
&=-\ri\Div\cX+\ri\epsilon_{ijk}\sigma_k\{-\ri\nabla_i-\magp_i,\cX_j\}\nonumber\\
\label{eq:commDQ2}
&=4\ri\sigma_k(\nabla_k\cKf_i)(-\ri\nabla_i-\magp_i)-\ri\sfrac43(\Div\cKf)\Dirac+\sfrac43\sigma.\nabla(\Div\cKf).
\end{align}
Note that $[\Dirac,\Div\cKf]=-\ri\sigma.\nabla(\Div\cKf)$. 
Combined with \eqref{eq:commDQ1} and \eqref{eq:commDQ2} we now get
\[
[\Dirac,\OpQ]=-\sfrac13\sigma.\nabla(\Div X)-\ri\sfrac13(\Div X)\Dirac
=-\ri\sfrac13\Dirac\bigl((\Div\cKf)\,\cdot\,\bigr).
\]
However
\[
[\OpQ,\wei]=-\ri\nabla_\cKf\wei=-\ri\sfrac13(\Div\cKf)\wei
\]
by \eqref{eq:XDabsXa}.
Combining the previous two equations we finally get
\[
[\Dirac\wei,\OpQ]=[\Dirac,\OpQ]\wei-\Dirac\bigl([\OpQ,\wei]\,\cdot\,\bigr)
=-\ri\sfrac13\Dirac\bigl((\Div\cKf)\wei-(\Div\cKf)\wei\bigr)
=0.
\]

With the help of \eqref{eq:XDabsXa} 
\begin{equation}
\label{eq:commQS1}
[\cKf.(-\ri\nabla-\magp),\OpS]
=-\ri\bigl[\nabla_\cKf,\wei[-1]\sigma.\cKf\bigr]
=\ri\sfrac13(\Div\cKf)\wei[-1]\sigma.\cKf-\ri\wei[-1]\sigma_i\nabla_\cKf\cKf_i.
\end{equation}
On the other hand, \eqref{eq:PauliMatId} and \eqref{eq:XxCurlX} lead to
\[
\label{eq:commQS2}
[\sigma.\cX,\OpS]
=2\ri\wei[-1]\sigma.(\cX\times\cKf)
=4\ri\wei[-1]\sigma_i\nabla_\cKf\cKf_i-\ri\sfrac43(\Div\cKf)\OpS.
\]
Since $\OpS$ clearly commutes with (multiplication by) $\Div\cKf$ 
we can now combine this with \eqref{eq:commQS1} and \eqref{eq:commQS2} to get $[\OpQ,\OpS]=0$.

For the final part \eqref{eq:PauliMatId} helps give
\begin{align*}
\{\OpS,\Dirac\wei\}
&=\wei[-1]\sigma_i\sigma_j\cKf_i(-\ri\nabla_j-\magp_j)(\wei\,\cdot\,)
+\sigma_j\sigma_i(-\ri\nabla_j-\magp_j)\bigl(\cKf_i\,\cdot\,\bigr)\\
&=(\sigma_i\sigma_j+\sigma_j\sigma_i)\cKf_i(-\ri\nabla_j-\magp_j)
-\ri\wei[-1]\sigma_i\sigma_j\cKf_i\nabla_j\wei
-\ri\sigma_j\sigma_i\nabla_j\cKf_i\\
&=2\cKf.(-\ri\nabla-\magp)
-\ri\wei[-1]\nabla_\cKf\wei-\ri\Div\cKf
+\wei[-1]\sigma.(\cKf\times\nabla\wei)+\sigma.\cX\\
&=2\cKf.(-\ri\nabla-\magp)
-\ri\sfrac{4}{3}\Div\cKf
+\sfrac12\sigma.\cX+\sfrac12\wei[-1]\cKf.\cX\OpS,
\end{align*}
where \eqref{eq:XDabsXa} and \eqref{eq:altCurlX} have been used in the last step.
\end{proof}

The commutator relations given in Proposition \ref{prop:baspropDQS}
lead naturally to a weighted version of $L^2$. Let $\wLspin$ denote 
the set of spinors $\phi:\R^3\to\C^2$ satisfying
\[
\int\abs{\phi(\bfx)}^2\wei(\bfx)\,\rd^3\bfx<+\infty.
\]
We denote the corresponding norm and inner product by $\wnorm{\cdot}$ and $\wipd{\cdot}{\cdot}$ respectively.
Note that $\cssfX$ is dense in $\wLspin$.

Since $\Dirac$ defined on $C^\infty_0$ is symmetric in $\Lspin$, 
$\Dirac\wei$ defined on $\cssfX$ is symmetric in $\wLspin$.
We can use \eqref{eq:XDabsXa} to help calculate the formal adjoint of $\OpQ$ in $\wLspin$ as
\begin{align*}
\OpQ^*&=\wei^{-1}(-\ri\nabla-\magp).(\wei\cKf\,\cdot\,)+\sfrac14\cX.\sigma+\sfrac23\ri\Div\cKf\\
&=\cKf.(-\ri\nabla-\magp)-\ri\wei[-1]\nabla_\cKf\wei-\ri\Div\cKf+\sfrac14\cX.\sigma+\sfrac23\ri\Div\cKf
=\OpQ.
\end{align*}
It follows that $\OpQ$ defined on $\cssfX$ is also symmetric in $\wLspin$. 
The operator $\OpS$ is bounded and self-adjoint on both $\Lspin$ and $\wLspin$.

Define operators
\[
\proj{\pm}=\sfrac12(I_2\pm\OpS).
\]
Clearly $\proj{+}+\proj{-}=I_2$ while the identity $\OpS^2=I_2$ means
$\proj{+}$ and $\proj{-}$ are pointwise projections on $\C^2$ 
as well as complementary orthogonal projections on $\wLspin$. 
Furthermore $\proj{\pm}\OpS=\pm\proj{\pm}$ while $\proj{\pm}$ commutes with multiplication
by scalar functions. On $\sfX$ Proposition \ref{prop:baspropDQS}(ii) gives $[\proj{\pm},\OpQ]=0$ 
while
\begin{equation}
\label{eq:PpmDPpmgen}
\proj{\pm}\Dirac\wei\proj{\pm}
=\sfrac14(\OpS\pm I_2)(\OpS\Dirac\wei+\Dirac\wei\OpS)
=\pm\sfrac12\proj{\pm}\{\OpS,\Dirac\wei\}.
\end{equation}
Let $\odD{\pm}=\proj{\pm}\Dirac\wei\proj{\mp}$ denote the ``off-diagonal'' components of $\Dirac$. 
By Proposition \ref{prop:baspropDQS}(i) 
\[
\odD{\mp}\proj{\pm}\OpQ
=\proj{\mp}\Dirac\wei\OpQ\proj{\pm}
=\proj{\mp}\OpQ\Dirac\wei\proj{\pm}
=\OpQ\proj{\mp}\odD{\mp}.
\]
For $\phi,\psi\in\cssfX$ the above observations now lead to 
$\wipd{\phi}{\odD{\pm}\psi}=\wipd{\odD{\mp}\phi}{\psi}$,
\begin{equation}
\label{eq:TPcommrel}
\wipd{\proj{\pm}\OpQ\phi}{\odD{\pm}\psi}
=\wipd{\odD{\mp}\proj{\pm}\OpQ\phi}{\psi}
=\wipd{\odD{\mp}\phi}{\proj{\mp}\OpQ\psi}
\end{equation}
and
\begin{equation}
\label{eq:mnormprojsplit}
\wnorm{\proj{+}\phi+\proj{-}\psi}^2
=\wnorm{\proj{+}\phi}^2+\wnorm{\proj{-}\psi}^2.
\end{equation}

\medskip

The addition of condition \eqref{eq:simrotcond} allows an important 
simplification in Proposition \ref{prop:baspropDQS}(iii); this leads to the following.

\begin{prop}
\label{prop:normspindecomp}
Suppose $\cKf$ satisfies \eqref{eq:simrotcond}.
For all $\phi\in\cssfX$ we have
\[
\wnorm{\Dirac\wei\phi}^2
=\wnorm{\odD{+}\phi}^2+\wnorm{\odD{-}\phi}^2+\wnorm{\OpQ\phi}^2.
\]
\end{prop}

\begin{proof}
By Proposition \ref{prop:baspropDQS}(iii) $\{\OpS,\Dirac\wei\}\phi=2\OpQ\phi$. 
Since $\proj{+}+\proj{-}=I_2$ and $\proj{\pm}^2=\proj{\pm}$, \eqref{eq:PpmDPpmgen} now leads to
\begin{align*}
\Dirac\wei\phi&=\proj{+}\Dirac\wei\proj{+}\phi+\proj{+}\Dirac\wei\proj{-}\phi+
\proj{-}\Dirac\wei\proj{+}\phi+\proj{-}\Dirac\wei\proj{-}\phi\\
&=\proj{+}\bigl(\OpQ\phi+\odD{+}\phi\bigr)
+\proj{-}\bigl(\odD{-}\phi-\OpQ\phi\bigr).
\end{align*}
Using \eqref{eq:mnormprojsplit} it follows that
\begin{align*}
\wnorm{\Dirac\wei\phi}^2
&=\wnorm{\proj{+}\OpQ\phi+\odD{+}\phi}^2
+\wnorm{\odD{-}\phi-\proj{-}\OpQ\phi}^2\\
&=\wnorm{\odD{+}\phi}^2+\wnorm{\odD{-}\phi}^2
+\wnorm{\proj{+}\OpQ\phi}^2+\wnorm{\proj{-}\OpQ\phi}^2\\
&\qquad{}
+2\re\wipd{\proj{+}\OpQ\phi}{\odD{+}\phi}-2\re\wipd{\odD{-}\phi}{\proj{-}\OpQ\phi}.
\end{align*}
The terms in the final line cancel by \eqref{eq:TPcommrel}.
The result now follows from \eqref{eq:mnormprojsplit}.
\end{proof}

\section{Eigenvalues of $\OpQ$}

Although we have not defined $\OpQ$ as a(n unbounded) self-adjoint operator
we are able to limit potential eigenvalues when the integral curves of $\cKf$ are closed loops.

\begin{prop}
\label{prop:specQ}
Suppose the conformal Killing field $\cKf$ satisfies \eqref{eq:simrotcond} and has a closed integral curve $\gamma$ of period $\tau$.
Also suppose $\OpQ\phi=\lambda\phi$ for some $\phi\in\sfX$ and constant $\lambda$.
If $\phi$ is non-trivial on $\gamma$ then $\lambda\in(2\Z+1)\tau^{-1}\pi$; in particular, $\lambda\neq0$.
\end{prop}

\begin{proof}
From the discussion after \eqref{eq:altCurlXcp} we know that $\gamma$ lies in a plane $\pP$ with normal $\Nvec$
where $\Nvec=\icB(\gamma(t))$ for any $t$. 
Choose $\ues{0}\in\C^2$ with $\sigma.\Nvec\ues{0}=\ues{0}$ and $\abs{\ues{0}}^2=2$. Set $\ues{\pm}=\proj{\pm}\ues{0}$.
For any $\Pvec\in\R^3$ \eqref{eq:PauliMatId} leads to
\begin{equation}
\label{eq:e0s.Ze0gen}
\ipd{\ues{0}}{\sigma.\Pvec\ues{0}}
=\sfrac12\bigipd{\ues{0}}{(\sigma.\Nvec\sigma.\Pvec+\sigma.\Pvec\sigma.\Nvec)\ues{0}}
=2\Nvec.\Pvec.
\end{equation}
We now work on (the image of) $\gamma$.
Then $\ipd{\ues{0}}{\OpS\ues{0}}=0$ by \eqref{eq:simrotcond} and \eqref{eq:e0s.Ze0gen}, so
\[
\abs{\ues{\pm}}^2
=\ipd{\ues{0}}{\proj{\pm}\ues{0}}
=\sfrac12\ipd{\ues{0}}{(I\pm\OpS)\ues{0}}
=1.
\]
As $\XxcX.\cX=0$ \eqref{eq:PauliMatId} and \eqref{eq:e0s.Ze0gen} also lead to 
$\ipd{\ues{0}}{\OpS\sigma.\cX\ues{0}}=\wei[-1]\ipd{\ues{0}}{(\cKf.\cX+i\sigma.\XxcX)\ues{0}}=0$.
Combined with a further use of \eqref{eq:e0s.Ze0gen} we then get
\[
\ipd{\ues{0}}{\proj{\pm}\sigma.\cX\ues{0}}
=\sfrac12\ipd{\ues{0}}{(I\pm\OpS)\sigma.\cX\ues{0}}
=\Nvec.\cX
=\abs{\cX}.
\]
Since $\OpQ\ues{\pm}=\proj{\pm}\OpQ\ues{0}$ while $\ues{0}$ is constant we now have
\begin{align*}
\ipd{\ues{\pm}}{\OpQ\ues{\pm}}
&=\sfrac14\ipd{\ues{0}}{\proj{\pm}\sigma.\cX\ues{0}}
-\Bigl(\sfrac23i\Div\cKf+\cKf.\magp\Bigr)\ipd{\ues{\pm}}{\proj{\pm}\ues{0}}\\
&=\sfrac14\abs{\cX}-\sfrac23i\Div\cKf-\cKf.\magp.
\end{align*}

Set $\phi_\pm=\proj{\pm}\phi$. Since $[\OpQ,\OpS]\phi=0$ by Proposition \ref{prop:baspropDQS}(ii)
we have $\OpQ\phi_\pm=\lambda\phi_\pm$.
At any $x\in\nzX$ the range of $\proj{\pm}$ in $\C^2$ is a $1$-dimensional space 
spanned by the non-zero vector $\ues{\pm}$.
Thus we can write $\phi_{\pm}(\gamma(t))=u_{\pm}(t)\,\ues{\pm}(\gamma(t))$ for some $u_\pm:\R\to\C$.
Then $\OpQ\phi_\pm=-i(\rd/\rd t)u_\pm\,\ues{\pm}+u_\pm\OpQ\ues{\pm}$ so
\[
\ipd{\ues{\pm}}{\OpQ\phi_\pm}
=-i\frac{\rd}{\rd t}u_\pm+\Bigl(\sfrac14\abs{\cX}-\sfrac23i\Div\cKf-\cKf.\magp\Bigr)u_\pm.
\]
On the other hand $\lambda\phi_\pm=\lambda u_\pm\ues{\pm}$ so 
$\ipd{\ues{\pm}}{\lambda\phi_\pm}=\lambda u_\pm\abs{\ues{\pm}}^2=\lambda u_\pm$. 
The equation $\OpQ\phi_\pm=\lambda\phi_\pm$ thus gives
\[
\frac{\rd}{\rd t}u_\pm+i\Bigl(\sfrac14\abs{\cX}-\sfrac23i\Div\cKf-\cKf.\magp-\lambda\Bigr)u_\pm=0.
\]
Since $\gamma$ is a closed integral curve with period $\tau$ it follows that we must have
\[
\int_0^\tau\Bigl(\sfrac14\abs{\cX}-\sfrac23i\Div\cKf-\cKf.\magp-\lambda\Bigr)(\gamma(t))\,\rd t\in 2\pi\Z.
\]
The result now follows from \eqref{eq:0intDivXloop}, \eqref{eq:HUevalint} and \eqref{eq:cloloopflux0}.
\end{proof}

\section{Regularity of zero modes}
\label{sec:regZMs}

Suppose $\magp$ is a smooth magnetic potential.
Then $\Dirac$ is essentially self-adjoint on $C^\infty_0$ (see \cite[Theorem 4.3]{T} for example); 
in particular, for any $\psi\in\Dom(\Dirac)$ we can find $\{\psi_n\}_{n\in\N}\subset C^\infty_0$ 
with $\psi_n\to\psi$ and $\Dirac\psi_n\to\Dirac\psi$ in $\Lspin$. 
However $\Dirac$ is also an elliptic operator; if $\Dirac\psi\in C^\infty$ for some $\psi\in\Dom(\Dirac)$ 
(which is clearly the case for zero modes) elliptic regularity then gives $\psi\in C^\infty$.
In this case it is clear from the proof of \cite[Theorem 4.3]{T} that we may further assume
our approximating sequence satisfies $\psi_n\to\psi$ pointwise on $\R^3$.

Now suppose $\cKf$ is a conformal Killing field satisfying \eqref{eq:simrotcond}. 
If $\psi$ is a potential zero mode we need a more particular approximating sequence 
that has $\psi_n\in\cssfX$ and $\Dirac\psi_n\to\Dirac\psi$ in $\wLspin$. We begin by considering these 
conditions separately.

Choose a non-increasing $\pcf[]\in C^\infty(\R)$ with $\pcf[](t)=1$ for $t<0$ and $\pcf[](t)=0$ for $t>1$.
For any $\Rc>1$ define $\pcf\in C^\infty(\R)$ by
\[
\pcf(t)=\begin{cases}
1&\text{if $t\le1$,}\\
\pcf[](\log\log t-\log\log\Rc)&\text{if $t>1$.}
\end{cases}
\]
Also define $\Rcf\in C^\infty_0$ by $\Rcf(\bfx)=\pcf(\abs{\bfx})$.

\begin{lem}
\label{lem:RcfpsiwL}
Suppose $\psi\in\Lspin\cap C^\infty$ and set $\psi_{\Rc}=\Rcf\psi\in C^\infty_0$ for $\Rc>1$.
If $\Dirac\psi\in\wLspin$ then $\Dirac\psi_{\Rc}\to\Dirac\psi$ in $\wLspin$ as $\Rc\to\infty$.
\end{lem}

\begin{proof}
Note that $\supp(\nabla\Rcf)=\{\bfx:\Rc\le\abs{\bfx}\le\Rc^e\}$ while, for $\abs{\bfx}>1$, 
\[
\nabla\Rcf(\bfx)
=\frac{\bfx}{\abs{\bfx}^2\log\abs{\bfx}}\,\pcf[]'(\log\log\abs{\bfx}-\log\log\Rc).
\]
As $\wei[1/2](\bfx)\le C(1+\abs{\bfx})$ for some constant $C$ we then get
\begin{equation}
\label{eq:LinfrDchi0}
\norm{\wei[1/2]\nabla\Rcf}_{L^\infty}\le2C\norm{\pcf[]'}_{L^\infty}\sup_{\Rc\le\abs{\bfx}\le\Rc^e}\frac1{\log\abs{\bfx}}
=\frac{2C\norm{\pcf[]'}_{L^\infty}}{\log\Rc}\longrightarrow 0
\end{equation}
as $\Rc\to\infty$. Now $\abs{\sigma.\nabla\Rcf\vtsp\psi}=\abs{\nabla\Rcf}\vtsp\abs{\psi}$ so
\begin{align*}
&\int\abs{\Dirac(\psi_{\Rc}-\psi)(\bfx)}^2\wei(\bfx)\,\rd^3\bfx\\
&\qquad{}\le2\int\abs{\nabla\Rcf(\bfx)}^2\abs{\psi(\bfx)}^2\wei(\bfx)\,\rd^3\bfx
+2\int\abs{\Rcf(\bfx)-1}^2\abs{\Dirac\psi(\bfx)}^2\wei(\bfx)\,\rd^3\bfx.
\end{align*}
Both terms in the final line tend to $0$ as $\Rc\to\infty$, the first by \eqref{eq:LinfrDchi0}
and the second by dominated convergence (note that $\Rcf(\bfx)\to1$ for all $\bfx$).
\end{proof}

For any $\epsilon\in(0,1)$ define $\ecf\in C^\infty$ by $\ecf(\bfx)=\pcf[1/\epsilon](\wei^{-1}(\bfx))$. 
Thus $\ecf(\bfx)\to1$ as $\epsilon\to0$ for any $\bfx\in\nzX$. Also let $H^1$ denote the first order Sobolev space 
(consisting of functions with $\psi,\nabla \psi\in L^2$).

\begin{lem}
\label{lem:fzXdet}
If $\cKf$ is an admissible field and 
$\psi\in C^\infty_0$ then $\ecf\psi\to\psi$ in $H^1$ as $\epsilon\to0$.
\end{lem}

\begin{proof}
We have
\[
\nabla\ecf
=(\nabla\wei[-1])\,\pcf[1/\epsilon]'(\wei[-1])
=\frac{\nabla\wei}{\wei\log\wei}\,\pcf[]'(\log\log\wei-\log\log\epsilon).
\]
In particular, $\supp(\nabla\ecf)\subseteq\Delta_{\epsilon}:=\{\bfx:\epsilon^e\le\abs{\wei(\bfx)}\le\epsilon\}$,
while 
\[
\abs{\nabla\ecf}^2\le\abs{\nabla\wei}^2\,\ell(\wei)\,\norm{\pcf[]'}_{L^\infty}
\]
where $\ell(r)=1/(r^2\log^2r)$ for $r>0$ and $\abs{\nabla\wei}^2$ is a smooth function (see \eqref{eq:simexprDX2}).
For any admissible field we have $\ell(\wei)\in L^1_\loc$,
while $\abs{\Delta_\epsilon\cap K}\to0$ as $\epsilon\to0$ for any compact set $K$
(note that $\{\bfx:\wei(\bfx)=0\}$ is either empty or a $1$-dimensional sub-manifold of $\R^3$).  
Now
\[
\int\abs{\nabla(\ecf\psi-\psi)(\bfx)}^2\,\rd^3\bfx
\le2\int\abs{\nabla\ecf(\bfx)}^2\abs{\psi(\bfx)}^2\,\rd^3\bfx
+2\int\abs{\ecf(\bfx)-1}^2\abs{\nabla\psi(\bfx)}^2\,\rd^3\bfx.
\]
On the right hand side both terms tend to $0$ as $\epsilon\to0$, the first from the discussion above
and the second by dominated convergence.
\end{proof}

\begin{rem}
The function $\Rcf$ provides a cut-off for large $\bfx$, while $\ecf$ provides a cut-off near $\{\bfx:\wei(\bfx)=0\}$.
In particular, Lemma \ref{lem:fzXdet} shows that the set $\{\bfx:\wei(\bfx)=0\}$ has (harmonic) capacity $0$, or is $(-1,2)$-null;
in $\R^3$ $1$-dimensional submanifolds are in some sense borderline in this respect 
(see \cite{HM} for further discussion).
\end{rem}

\begin{prop}
\label{prop:approxseqpsi}
Suppose $\psi\in\Lspin\cap C^\infty$. If $\Dirac\psi\in\wLspin$ 
we can find $\{\psi_n\}_{n\in\N}\subset\cssfX$ so that $\psi_n\to\psi$ in $\Lspin$, 
$\psi_n(\bfx)\to\psi(\bfx)$ for $\bfx\in\nzX$, and $\Dirac\psi_n\to\Dirac\psi$ in $\wLspin$.
\end{prop}

\begin{proof}
Elliptic regularity gives $\psi\in C^\infty$. 
Set $\psi_1\equiv0$ and $\psi_n=\Rcf[n]\ecf[\epsilon_n]\psi\in\cssfX$ for $n>1$, 
where $\epsilon_n\to0$ but otherwise remains to be chosen.
Then $\psi_n\to\psi$ in $\Lspin$ while $\psi_n(\bfx)\to\psi(\bfx)$ for $\bfx\in\nzX$. 
Now set 
\[
M_n=\sup\bigl\{4(1+\abs{\magp(\bfx)}^2)\wei(\bfx):\bfx\in\supp(\Rcf[n])\bigr\}
\quad\text{and}\quad
\delta_n=M_n^{-1}2^{-n},
\]
so $\delta_n>0$ and $\delta_n\to0$.
Note that $\Rcf[n]\psi\in C^\infty_0$; we can then use 
Lemma \ref{lem:fzXdet} to find $\epsilon_n$ so that $\norm{\phi_n}_{H^1}^2\le\delta_n$,
where $\phi_n=\ecf[\epsilon_n]\Rcf[n]\psi-\Rcf[n]\psi=\psi_n-\Rcf[n]\psi$. Clearly we may further assume $\epsilon_n\to0$. 
Now $\big.\abs{\sigma.\nabla\phi_n}^2\le3\abs{\nabla\phi_n}^2$ and $\abs{\sigma.\magp\phi_n}=\abs{\magp}\vtsp\abs{\phi_n}$ so 
\[
\abs{\Dirac\phi_n}^2\le\sfrac43\abs{\sigma.\nabla\phi_n}^2+4\abs{\sigma.\magp\phi_n}^2
\le4(1+\abs{\magp}^2)\vtsp(\abs{\nabla\phi_n}^2+\abs{\phi_n}^2).
\]
Hence
\begin{align*}
\wnorm{\Dirac(\psi_n-\Rcf[n]\psi)}^2
&\le\int4(1+\abs{\magp(\bfx)}^2)\bigl(\abs{\nabla\phi_n(\bfx)}+\abs{\phi_n(\bfx)}^2\bigr)\wei(\bfx)\,\rd^3\bfx\\
&\le M_n\norm{\phi_n}_{H^1}^2\le 2^{-n},
\end{align*}
so $\Dirac(\psi_n-\Rcf[n]\psi)\to0$ in $\wLspin$.
However $\Dirac(\Rcf[n]\psi)\to\Dirac\psi$ in $\wLspin$ by Lemma \ref{lem:RcfpsiwL}. 
This completes the result.
\end{proof}

\section{Proof of the main result}
\label{sec:pfthm1}

Suppose $\cKf$ is an admissible field and $\psi\in C^\infty$. 
Then $\wei^{-1}\psi,\,\OpQ(\wei^{-1}\psi)\in\sfX$ while
\[
\OpQ(\wei^{-1}\psi)
=\wei^{-1}\Bigl(-\ri\nabla_\cKf\psi+\sfrac14\sigma.\cX\psi
-\sfrac13i\Div\cKf\,\psi-\cKf.\magp\,\psi\Bigr)
\]
(recall \eqref{eq:XDabsXa}); note that, the term in the parentheses on the right hand side is in $C^\infty$.
Since $\wei[-1]\in L^1_\loc$ for any admissible field, we then get
\begin{equation}
\label{eq:Qpsi/wL2loc}
\wei^{-1}\psi,\,\wei\abs{\OpQ(\wei[-1]\psi)}^2\in L^1_\loc.
\end{equation}

\begin{proof}[Proof of Theorem \ref{thm:mainthm}]
Suppose $\psi\in\Lspin$ is a zero mode of $\Dirac$. 
Elliptic regularity gives $\psi\in C^\infty$. 
Let $\{\psi_n\}_{n\in\N}\subset\cssfX$ be as given by Proposition \ref{prop:approxseqpsi}. 
Set $\phi=\wei[-1]\psi\in\sfX$ and, for each $n\in\N$, $\phi_n=\wei[-1]\psi_n\in\cssfX$. 
We immediately have $\phi_n(\bfx)\to\phi(\bfx)$ and $\OpQ\phi_n(\bfx)\to\OpQ\phi(\bfx)$ for $\bfx\in\nzX$, 
while $\Dirac\wei\phi_n=\Dirac\psi_n\to0$ in $\wLspin$. 
From \eqref{eq:Qpsi/wL2loc} we also have $\wei\abs{\OpQ\phi_n}^2,\,\wei\abs{\OpQ\phi}^2\in L^1_\loc$.
However Proposition \ref{prop:normspindecomp} gives
\[
\int\abs{\OpQ\phi_n(\bfx)}^2\wei(\bfx)\,\rd^3\bfx=
\wnorm{\OpQ\phi_n}^2\le\wnorm{\Dirac\wei\phi_n}^2
=\int\abs{\Dirac\psi_n(\bfx)}^2\wei(\bfx)\,\rd^3\bfx
\]
for all $n\in\N$. Taking $n\to\infty$ and applying Fatou's lemma it follows that 
\[
\int\abs{\OpQ\phi(\bfx)}^2\wei(\bfx)\,\rd^3\bfx=0.
\]
Hence $\wei^{1/2}\OpQ\phi=0$ almost everywhere. Since $\OpQ\phi\in\sfX$ we then get $\OpQ\phi=0$ on $\nzX$.
If $\cKf$ is non-constant the closed integral curves of $\cKf$ are dense (in $\R^3$). Since $\phi$ is non-trivial 
we then obtain a contradiction with Proposition \ref{prop:specQ}. 

To deal with the remaining case we may assume $\cKf=\euv{3}$. Then
$\nzX=\R^3$ while the equation $\OpQ\phi=0$ becomes
\[
(-\ri\nabla_3-A_3)\phi=0
\quad\Longleftrightarrow\quad
\nabla_3(e^{-\ri f}\phi)=0,
\]
where $f\in C^\infty$ is any real-valued function with $\nabla_3f=A_3$. It is straightforward to see 
that this final equation cannot have a non-trivial solution $\phi\in\Lspin\cap C^\infty$. 
\end{proof}

\bigskip

\noindent
Daniel M.\ Elton\\
Department of Mathematics and Statistics\\
Fylde College\\
Lancaster University\\
Lancaster LA1 4YF\\
United Kingdom

\smallskip

\noindent
e-mail: \texttt{d.m.elton@lancaster.ac.uk}

\end{document}